\documentclass[12pt,epsfig,tikz,multi]{amsart}
\usepackage{url}
\usepackage{float}
\restylefloat{table}
\usepackage{verbatim}
\usepackage{latexsym}
\usepackage{hyperref}
\usepackage{amsrefs}
\usepackage{amsthm}
\usepackage{breqn}
\usepackage{float}
\usepackage{booktabs}
\usepackage{multirow}
\usepackage{amsmath,amscd,amssymb}
\usepackage{graphicx}
\usepackage{caption}
\usepackage{subcaption}
\usepackage{multirow}
\usepackage{enumitem}
\usepackage{xcolor}
\usepackage{wrapfig}
\usepackage{lipsum}
\usepackage{tikz}
\usepackage{diagbox}
\usepackage{float}
\usepackage{import}
\usepackage{xifthen}
\usepackage{pdfpages}
\usepackage{transparent}
\usepackage{amsmath,amssymb,mathtools}

 2

\newtheorem{theorem}{Theorem}[section]
\newtheorem{prop}{Proposition}[section]

\newtheorem{definition}[theorem]{Definition}

\numberwithin{equation}{section}
\def \C {{\mathbb {C}}}
\def \Z {{\mathbb {Z}}}

\newcommand{\re}{\operatorname{Re}}

\title{Maxfaces with infinitely many Swallowtails}
\author{Anu Dhochak}
\address{International Centre for Theoretical Sciences (ICTS-TIFR), Bengaluru 560089, Karnataka, India.}
\email{anu.dhochak@icts.res.in}
\subjclass[2020]{53A35}
\keywords{zero mean curvature surfaces, embedded maxface, maxface with infinite ends, swallowtails.}
\begin{document}

\maketitle
\begin{abstract}
    In this article, we discuss the existence of a 1-parameter infinite genus family of maxfaces having infinitely many planar (spacelike) ends and infinitely many swallowtails.   In particular, we show the existence of :
    \begin{enumerate}
    \item  a period-2 family of maxface with infinitely many planar ends and alternating singularity types—every odd-layer neck has exactly four swallowtails, while each even-layer neck is almost conical, for $n=2$ the fundamental piece is a genus-1 Wei-type maxface;
    \item a period-3 family where every neck carries four swallowtails (24 per period);  
    \item a period-2 family of maxface with an almost conical singularity on every neck.
    \end{enumerate}
    All families of maxfaces are embedded in a wider sense. 
\end{abstract}

\section{Introduction}
Maximal surfaces are zero–mean–curvature spacelike immersions in the Lorentz--Minkowski space $\mathbb{E}^{3}_{1}$. Like their Euclidean minimal counterparts, they arise as critical points of the area functional and admit Weierstrass–Enneper representations. In sharp contrast with the abundance of complete minimal surfaces in $\mathbb{R}^{3}$, the only complete maximal immersion in $\mathbb{E}^{3}_{1}$ is the plane \cite{UMEHARA2006}. 

To overcome this rigidity, one allows admissible singularities and works with \emph{maximal maps} and, more specifically, \emph{maxfaces} in the sense of Umehara–Yamada \cite{UMEHARA2006}. Complete non–planar maxfaces necessarily have a compact singular set, and embeddedness has to be understood in the wider sense (embedded outside a compact set); we adopt this convention throughout (cf. \cite{UMEHARA2006,Kim2006}). 

Classical and recent examples already show a rich landscape: the Lorentzian catenoid and the Kim–Yang toroidal example \cite{Kim2006}; the families in \cite{Fujimori2016highergenus}; and related constructions of complete maxfaces with few ends. Nevertheless, higher–genus maxfaces in the literature typically had only two or three (catenoidal) ends, and very few are embedded in the wider sense. 
\begin{sloppypar}
Recently, Chen-Dhochak-Kumar-Mohanty \cite{chennode}, adapted  Traizet’s \emph{node–opening} technique \cite{traizet2002adding,traizet2002embedded}, Weierstrass gluing framework to the Lorentzian setting and produced many arbitrary finite genus families of maxfaces with many spacelike ends, including Lorentzian Costa and Lorentzian Costa–Hoffman–Meeks type examples. Their work demonstrated that the balance/non-degeneracy conditions for the node–opening technique coincide with those in the minimal case.  
\end{sloppypar}
\medskip

The present article pushes the node–opening method beyond the finite–genus regime under natural boundedness and balancing hypotheses (in the spirit of Morabito–Traizet’s infinite–genus minimal surfaces \cite{morabito2012non}), and in Section \ref{sec:example}, we give many examples of infinite genus maxfaces.   From \cite{chennode}, it is clear that \emph{infinite} node–opening scheme will work verbatim to construct a real $1$–parameter family of complete maxfaces that are embedded (in the wider sense), have \emph{infinite genus}, and carry \emph{infinitely many}  spacelike planar ends. Since the calculations follow the same pattern, we do not explicitly rewrite the proofs. However, the final results are interesting; we don't see it in the literature, so in section \ref{sec:existence} we state the existence results with the necessary modifications required. 

On these new infinite genus maxfaces, as in \cite{chennode}, the singular set consists of curves located along the waist of each neck. To determine the nature of singularities, in \cite{chennode}, the authors introduced the functions $R_{l,k}^{(r)}(\theta)$ and analyzed the behavior of this function. Here, in our infinite-genus case, due to slight modifications in the node–opening process, we have to modify the functions to $R^{(r)}_{k,i}(\theta)$ accordingly.  In Subsection \ref{sec:singularities}, we calculate it and recall all related singularities as in \cite{chennode}.

We see many interesting behaviors of singularities in the infinite genus case in compare to the finite genus case as in \cite{chennode}, in particular, in \cite{chennode}, the singularities were analysed when for some $r>0$, $R_{l,k}^{(r)}(\theta)\not\equiv 0$  with $\theta$ identifying a singularity. In contrast, in Proposition ~\ref{prop1}, we show that $R^{(r)}_{k,i}(\theta) \equiv 0$ for all $r \in \mathbb{N}$ for certain configurations.

We introduce the notion of an \emph{almost–conical} singularity, defined by the vanishing \(R^{(r)}_{k,i}(\theta) \equiv0\) for all \(r\). This helps us single out the non-collapsed, highly isotropic case where the waist (singular curve) stays for \(t>0\) but the \(\theta\)-dependence disappears at \emph{every} order in \(t\). In simple terms, the singularity behaves like a cone to infinite order, without creating a true cone tip. This notion is natural for our periodic and quasi-periodic families, where such waists appear. We will see such examples in Section \ref{sec:example}.

We summarize the main findings of this article.  We give explicit constructions of periodic/quasi-periodic maxfaces with infinitely many planar ends with a mix of singularities.  More precisely, we have   
\begin{enumerate}
     \item \textbf{Period-$2$ family with alternating swallowtails and almost conical singularities-} We give a family of complete, embedded (in the wider sense) maxfaces with infinitely many planar ends. In this family, every neck in an odd-numbered layer produces exactly four swallowtails, while each neck in an even-numbered layer is almost conical. For $n=2$, the fundamental piece is a  \emph{genus-$1$ Wei-type} maxface. 
  \item \textbf{Period-$3$ family with four swallowtails on each neck-} We obtain a complete, embedded (in the wider sense), even genus family of maxfaces of period $3$. A direct computation shows that each of the six necks in one period carries four swallowtails, i.e., \emph{$6\times4=24$} swallowtails per period, hence \emph{infinitely many} on the whole surface.
  \item \textbf{Every neck with an almost conical singularity-} It is a family of zero genus maxafces where every layer contains only one neck with \emph{an almost conical} singularity. Geometrically, the picture strongly resembles the classical \emph{Riemann-type} maximal surfaces \cite{lopez2000maximal}, now in a periodic setting.
\end{enumerate}

\medskip

\noindent\textbf{Organization.}  
Section~\ref{sec:existence} recalls the necessary preliminaries on maxfaces and the node–opening setup. Furthermore, it carries out the infinite node–opening construction and solves the period problem.  Proof is verbatim, so only the sketch is given. 

Subsection~\ref{sec:singularities} analyzes the singularity set and establishes the generic emergence of swallowtails and an almost conical one. 

Finally, in Section~\ref{sec:example}, using the configuration given in \cite{morabito2012non},  we analyze the corresponding maxface and their singularities. 

\section*{Acknowledgment}
I would like to express my sincere gratitude to Prof. Pradip Kumar for his support and mathematical discussions throughout this work.


\section{Maxface: existence  and singularity part}\label{sec:existence}
We first recall the classical Weierstrass representation for maxfaces \cite{UMEHARA2006}. Let $M$ be a Riemann surface (possibly punctured at the ends), let $g$ be a meromorphic function on $M$, and let $dh$ be a holomorphic $1$-form on $M$ satisfying:
\begin{itemize}
\item \emph{Divisor condition:} Away from the punctures, we must have
    \(        (g)_0-(g)_\infty=(dh)_0.
    \)

\item \emph{Period Condition:} For all closed curves $C$ on $M$, we have
\begin{align}
        \int_{C}g^{-1}dh + \overline{\int_{C}gdh} &= 0, \label{eq:hperiod}\\
        \re\int_{C}dh &= 0. \label{eq:vperiod}
    \end{align}
\item \emph{Regularity condition:} $|g|$ is not identically $1$.
\end{itemize}
Then $X: M \to \mathbb{E}_1^3$ defined by
\begin{equation}\label{maximal_map}
    M \ni z \mapsto \re \int^z \Big(  \frac{1}{2}(g^{-1} + g), \frac{i}{2}(g^{-1} - g), 1 \Big) dh \in \mathbb{E}^3_1,
\end{equation}
gives a maxafce. Here $(M,\,g,\,dh)$ is called the Weierstrass data of the maxface. Conversely, for every maxface there exists a meromorphic function $g$ and a holomorphic 1-form $dh$ such that \eqref{maximal_map} represents it.

For our construction of an infinite genus family of maxfaces, we adopt the infinite configuration terminology of Morabito--Traizet \cite{morabito2012non} and solve the period problem for the configuration. We now recall the infinite configuration in the maxface settings.

\medskip
\noindent \textbf{Configurations and Forces:}
Let $(n_k)_{k \in \mathbb{Z}}$ be a sequence of positive integers, and let $(p_{k,i})_{k \in \mathbb{Z},\,1 \leq i \leq n_k}$ be a sequence of complex numbers. We refer to the sequence $(p_{k,i})_{ k \in \mathbb{Z},\,1 \leq i \leq n_k}$ as a \emph{configuration} of type $(n_k)_{k \in \mathbb{Z}}$. Define $c_k = \frac{1}{n_k}$.  For each $(k,i)$, we define the  \emph{force} 
\[
F_{k,i} :=\sum_{1 \le j \ne i \le n_k}\frac{2 c_k^2}{p_{k,i}-p_{k,j}}-\sum_{i=1}^{n_{k+1}}\frac{ c_kc_{k+1}}{p_{k,i}-p_{k+1,j}}-\sum_{i=1}^{n_{k-1}}\frac{ c_k c_{k-1}}{p_{k,i}-p_{k-1,j}},
\]
and assume that all the points $p_{k,\,i}$ and $p_{k\pm1,\,j}$ are pairwise distinct for each $k$.

Following Morabito--Traizet \cite{morabito2012non}, we introduce the following change of variables that includes translations and scalar multiplication: 
\begin{equation}\label{eq:u_il_k}
\begin{cases}
    u_{k,i} := (-1)^k (p_{k,i}-p_{k,1}),& 1 \leq i \leq n_k, k \in \mathbb{Z}, \\
    l_k := (-1)^k(p_{k,1}-p_{k-1,1}), & k\in \mathbb{Z}.
\end{cases}
\end{equation}
By definition, $u_{k,1}=0$ for all $k$. Let $\boldsymbol{U}$ denote the infinite sequence
\begin{equation}\label{eq:U}
\ldots,\ell_k,u_{k,2},\dots,u_{k,n_k},\ell_{k+1},u_{k+1,2},\ldots,,
\end{equation}
which encodes the configuration up to translation and scaling. In terms of the parameter $\boldsymbol{U}$, we rewrite the forces as
\[
F_{k,i}=\sum_{1 \le j \ne i \le n_k}\frac{2(-1)^k c_k^2}{u_{k,i}-u_{k,j}}-\sum_{j=1}^{n_{k+1}}\frac{(-1)^k c_kc_{k+1}}{u_{k,i}+ l_{k+1} + u_{k+1,j}}-\sum_{j=1}^{n_{k-1}}\frac{(-1)^k c_k c_{k-1}}{u_{k,i}+ l_k + u_{k-1,j}}.
\]
We also define the maps
\[
G_k := \sum_{i=1}^{n_k} \sum_{j=1}^{n_{k-1}} \frac{c_kc_{k-1}}{p_{k,i}-p_{k-1,j}} = \sum_{i=1}^{n_k} \sum_{j=1}^{n_{k-1}}
\frac{(-1)^kc_kc_{k-1}}{u_{k,i} + l_k + u_{k-1,j}}.
\]
In terms of $F_{k,i}$ and $G_k$ we say:
\begin{definition}
    A configuration is said to be \emph{balanced} if the sequence
\[
    \tilde{F} := (...,G_k,F_{k,2},...,F_{k,n_k},G_{k+1},F_{k+1,2},...,F_{k+1,k_{k+1}},...)
\] 
is zero, and we call it \emph{non-degenerate} if the differential of $\tilde{F}$ with respect to $\boldsymbol{U}$ is an isomorphism from $l^\infty(\mathbb{Z})$ to itself. 
\end{definition}

In terms of configuration, our existence result is as follows
\begin{theorem} \label{main theorem}
    Consider a balanced, non-degenerate configuration $\boldsymbol{U}$ of type $(n_k)_{k\in \mathbb{Z}}$ such that 
    \begin{enumerate}
        \item the sequence $(n_k)_{k\in \mathbb{Z}}$ is bounded,
        \item the sequence $\boldsymbol{U}$ takes a finite number of values,
        \item for all $k\in \mathbb{Z}$, $\frac{1}{n_k}\sum_{i=1}^{n_{k}} p_{k,i} \neq \frac{1}{n_{k-1}} p_{k-1,i}.$
    \end{enumerate}
    Then there exists a 1-parameter family $(M_t)_{0 < t < \epsilon}$ of complete, embedded (in a wider sense) maxfaces with infinitely many planar ends. 
\end{theorem}
The proof is a straightforward adaptation of \cite{morabito2012non} and was verified for the finite configuration by Chen, Dhochak, Kumar, and Mohanty in \cite{chennode}. Since it is verbatim, we omit the proof here and outline the Weierstrass data of the constructed surfaces.

\medskip

\noindent{\textbf{Riemann surface:}} Consider an infinite chain of Riemann spheres ${\overline{\C}_k}, \text{ for }{k\in\Z}$. In each copy $\overline{\C}_k$, we fix distinct finite points:
\[
a_{k,1}, \ldots, a_{k,n_k} \in \C_k, 
\qquad 
b_{k,1}, \ldots, b_{k,n_k} \in \C_{k+1}.
\]
where $\C_k=\overline{\C}_k\setminus{\infty_k}$. Identifying $a_{k,i}=b_{k,i}$ for each $i$ gives a nodal curve $\Sigma_0$. We then apply the classical node-opening construction \cite{traizet2002adding},\cite{morabito2012non}, which we outline below.

Define a meromorphic function on each copy of Riemann sphere $\overline{\C_k} = \C_k \cup \{\infty_k\}$ by 
\begin{equation}\label{eq:g_k}
g_k:=\sum_{i=1}^{n_{k-1}}\frac{\beta_{k-1,i}}{z-b_{k-1,i}}-\sum_{i=1}^{n_k}\frac{\alpha_{k,i}}{z-a_{k,i}},
\end{equation}
where the new parameters $\alpha:= (\alpha_{k,i})_{k\in \mathbb{Z}, 1\leq i \leq n_k} \text{ and } \beta:=(\beta_{k,i})_{k\in \mathbb{Z}, 1\leq i \leq n_k}$ are complex numbers.

Moreover, for each parameter we define initial value parameters $a^0,b^0,\alpha^0,\text{ and } \beta^0$ such that all parameters vary in a small neighbourhood of their initial value parameter. We also assume that all initial value parameter sequences $u^0 = (u^0_{k,i})_{k\in \Z, 1\leq i\leq n_k}$ follow finiteness hypotheses, i.e., the set $\{u^0_{k,i}\}$ is finite. In terms of the configuration data, define: 
\[
a^0_{k,i} = (-1)^k \text{conj}^k(l^0_k + u^0_{k,i}), \quad \text{ and } \quad b^0_{k,i}  = (-1)^k \text{conj}^{k+1} (u^0_{k,i}).
\]
Thus, in each copy of a complex plane for all $1\leq i \leq n_k$, and $1\leq j \leq n_{k-1}$, the central values corresponding to the parameters $a$ and $b$ are distinct.   We take variable $\alpha$ and $\beta$ those satisfying
\[
\forall k\in \Z, \quad \sum_{i=1}^{n_k}\alpha_{k,i} = \sum_{i=1}^{n_k} \beta_{k,i}=1,
\]  
with the central values
\[
\alpha^0 = (\alpha^0_{k,i})_{k \in \mathbb{Z}, 1\leq i \leq n_k} = 
 \beta^0 = (\beta^0_{k,i})_{k \in \mathbb{Z}, 1\leq i \leq n_k} = \frac{1}{n_k}.
\]

Define local  coordinates around the points $a_{k, i}$ and $b_{k, i}$ by $v_{k,i} = \frac{1}{g_k}$ and $w_{k,i}=\frac{1}{g_{k+1}}$ respectively. Using the finiteness hypothesis, we choose a sufficiently small $\rho>0$, which is independent of $\{k, i\}$. Now, for any $ t\in (0, \rho)$ and for each $k \in \mathbb{Z}$ and $1\leq i \leq n_k$, we remove the disks $|v_{k,i}| < t^2/\rho$  and $|w_{k,i}| < t^2/\rho$, around $a_{k,i}$ and $b_{k,i}$ respectively, and identify the annular regions
\[
    t^2/\rho < |v_{k,i}| < \rho  ,\qquad t^2/\varepsilon < |w_{k,i}| < \rho,
\]
by
\[
    v_{k,i}(z) w_{k,i}(z') = t^2.
\]
Geometrically, this creates a neck connecting $\overline{\C}_k$ and $\overline{\C}_{k+1}$. The resulting Riemann surface, we denote  by $\Sigma_t$.

\medskip
\noindent{\textbf{Gauss map:}}
The Gauss map $g$ on $\Sigma_t$ is given by 
\[
g(z) :=
\begin{cases}
    t g_k(z) & \text{if } z\in \overline{\C_k} \text{ and $k$ is even,}\\
    1/(t g_k(z)) & \text{if } z\in\overline{\C_k}\text{ and $k$ is odd.}
\end{cases}
\]

\medskip 

\noindent {\textbf{The height differential:}} Define a holomorphic form $\omega$ on $\Sigma_t$ by prescribing its periods $(\gamma_{k,i})_{k\in \Z,1\le i\le n_k}$. We take $\omega$ as in \cite{morabito2012non}.  In particular, we choose  $\omega$ so that all periods vanish and the $z$-coordinate of \eqref{maximal_map} is real-valued.

\medskip

The surfaces \((\Sigma_t,\, g,\, \omega)\) are defined using several variables, namely \(a,\, b,\, \alpha,\, \beta,\, \gamma\) and \(t\). 
In~\cite{chennode}, in a similar situation, it  was verified that for sufficiently small \(t>0\), one can choose \(a, b, \alpha, \beta\) and \(\gamma\) as functions of \(t\) so that the period problem is solved. 
This yields a family of maxfaces as stated in Theorem~\ref{main theorem}. 
In the next subsection, we analyze the singularities of this family of maxfaces.


\subsection{About Singularities}\label{sec:singularities}
The singular set of a maxface is characterized by the condition $|g| = 1$. In our family of maxfaces, the singular set is given by the union of connected curves represented by 
$$
\mathcal{S}_{k,i} \;:=\; \{ z \in \C_k : |v_{k,i}(z)| = t \} \;=\; \{ z \in \C_{k+1} : |w_{k,i}(z)| = t \}, 
$$
where $v_{k,i}$ and $w_{k,i}$ are the local coordinates on the two halves of the neck. Thus geometrically,  each singular point of a maxface lies on the ``waist'' of some neck.

The key tool used in~\cite{chennode} to analyze singularities is a family of real-valued functions $R^{(r)}_{k,i}(\theta)$ depending on an index $r\in \mathbb{N}$ and a neck labeled by $(k,i)$.     The main difference between the finite (as in \cite{chennode},\cite{traizet2002adding}) and the infinite configurations \cite{morabito2012non} is the location of the node.  In ~\cite{chennode}, given a finite configuration $(p_{i})$, the Riemann surface was obtained by removing the node at the points $p_{i}$, but in the infinite dimensional case, we remove the nodes at $a^0_{k,i}$ and $b^0_{k,i}$, which are related to the configuration $(p_{k,i})_{k\in\Z, 1\le i\le n_k}$ as follows: 
\begin{equation}\label{eq:a^0,b^0}
    a^0_{k,i} = \mathrm{conj}^k(p_{k,i} - p_{k-1,1}) \, \text{ and } \,
b^0_{k,i} = \mathrm{conj}^{k+1}(p_{k,i} - p_{k,1}).
\end{equation}
Consequently, the expression of $R^{(r)}_{l,k}$ in ~\cite{chennode} will change, and we write it here with proper modification.   In terms of the configuration $(p_{k,i})_{k\in \Z, 1\leq i\leq n_k}$ it is given by 
\begin{equation}\label{eq:R_{k,i}}
R^{(r)}_{k,i}(\theta)
:= 
\begin{dcases} 
\mathrm{Im}\!\Bigl(
  e^{(r+1)i\theta} 
  \overline{\mathrm{Res}_{p_{k,i}-p_{k-1,1}}\Big(\frac{\omega_k}{dz}\Big)^{r+2}}
  \;-\;
  e^{-(r+1)i\theta}
  \mathrm{Res}_{p_{k,i}-p_{k,1}}\Big(\frac{\omega_{k+1}}{dz}\Big)^{r+2}
\Bigr), & k\text{ odd},\\[6pt]
\mathrm{Im}\!\Bigl(
  e^{(r+1)i\theta}
  \mathrm{Res}_{p_{k,i}-p_{k-1,1}}\Big(\frac{\omega_k}{dz}\Big)^{r+2}
  \;-\;
  e^{-(r+1)i\theta}
  \overline{\mathrm{Res}_{p_{k,i}-p_{k,1}}\Big(\frac{\omega_{k+1}}{dz}\Big)^{r+2}}
\Bigr), & k\text{ even},
\end{dcases}
\end{equation}
where
\begin{equation}\label{eq:w_k}
  \omega_k \;=\;
\sum_{i=1}^{n_{k-1}}
  \frac{c_{k-1}}{z-(p_{k-1,i}-p_{k-1,1})}
\;-\;
\sum_{i=1}^{n_k}
  \frac{c_k}{z-(p_{k,i}-p_{k-1,1})}.  
\end{equation}

 With these functions $R_{k,i}^{(r)}(\theta)$, the proof of results in \cite[Section 2.2]{chennode} is verbatim.  We write the required points with proper modification for our setup (infinite genus). 
\begin{theorem}\label{thm:singular-structure}
For all sufficiently small nonzero \(t\), the singular set of \(M_t\) consists of infinite disjoint closed curves, one for each neck \((k,i)\).  Moreover:
\begin{enumerate}
  \item All singular points are nondegenerate (i.e.\ fronts, no cross‐caps).
\item If \(R^{(1)}_{k,i}(\theta)\not\equiv0\) at \(t=0\), then each neck \((k,i)\) carries exactly four swallowtail singularities.  
\item If a neck is fixed by a vertical reflection, the fixed singular point is non‐cuspidal (a swallowtail or higher \(A_{2m+1}\)); if fixed by a horizontal reflection, the entire waist collapses to a cone singularity.
\end{enumerate}
\end{theorem}

The preceding theorem describes the generic situation, namely that along each neck one obtains a closed singular curve and, unless a symmetry forces a collapse, the function $R^{(1)}_{k,i}(\theta)$ detects four swallowtails. However, in the infinite–genus/node–opening situation that we are using, there are natural configurations (see in particular the periodic and quasi–periodic ones constructed later) for which \emph{all} the higher–order quantities $R^{(r)}_{k,i}(\theta)$ vanish simultaneously for a given neck for every $r\in\mathbb N$ and every $\theta$. In such a case, the singular curve is still present for $t>0$, but the usual criterion “first nonzero $R^{(r)}_{k,i}$ determines the type” no longer distinguishes the singularity. To single out precisely this highly isotropic, non-collapsing regime—visible only in the infinite concatenations—we introduce the following notion.

\begin{definition}[Almost Conical singularity] \label{def:conical}
 Each point $\theta$ of a connected singular curve around the waist of a neck is called an almost conical singular point if $R^{(r)}_{k,i}(\theta) \equiv 0$ for all $\theta$ and for all $r$.
    
\end{definition}

Geometrically, for \(t>0\) we still get a small closed curve as the singular set, but its profile does not depend on \(\theta\) even to higher orders in \(t\). Thus, the singularity looks ``cone–like" in the whole Taylor expansion, but the waist actually stays of positive size and does \emph{not} collapse.

\section{Periodic Maxface with infinitely many swallowtails and almost conical singularities}\label{sec:example}
In this section, we give examples of maxfaces having infinite genus, infinitely many swallowtails, and other singularities.  These are constructed by concatenating a finite balanced and non-degenerate configuration into an infinite configuration.


In ~\cite{morabito2012non}, the authors define periodic and quasi-periodic configurations for a minimal surface.  Here we adapt those definitions for the case of maxface and recall them here. 

\begin{definition} A configuration of type $(n_k)_{k\in\Z}$ is called periodic if there exists a positive integer $T$ such that for all $k \in \Z, n_{k+T}= n_k, l_{k+T}=l_k$ and $u_{k+T,i}= u_{k,i}$ for all $2 \leq i \leq n_k$.  A maxface corresponding to a periodic configuration $\boldsymbol{U}$ is said to be a periodic maxface.
\end{definition}
\begin{definition}A maxface is said to be quasi-periodic if its configuration $\boldsymbol{U}$ is quasi-periodic, i.e., there exists a diverging sequence of integers $(T_n)_{n\in \mathbb{N}}$ such that for all $k \in \Z$, there exists an integer $N$ such that for $n \geq N$, 
\[n_{k+T_n}=n_k, \quad l_{k+T_n} = l_k \text{ and } u_{k+T_n,i} = u_{k,i} \text{ for }  2\leq i \leq n_k.
\]
\end{definition}

Let $h$ be a positive integer, call it height, and for this $h$, let $(n_0,n_1,...,n_h)$ be a finite sequence of necks such that $n_0 = n_h=1$ and let $(p_{k,i})_{0\leq k\leq h, 1\leq i\leq n_k}$ be a finite configuration of height $h$. The corresponding finite parameter sequence $\boldsymbol{U}$, as define in \eqref{eq:U}, is given by $(l_1,u_{1,2},...,u_{1,n_1},l_2,u_{2,2},...,l_h)$. In the following paragraph, as in \cite{morabito2012non}, we define an infinite balanced and non-degenerate configuration $\boldsymbol{U}$ of type $(n_k)_{k\in \Z}=(...,n_{h-1},1,n_1,...,n_{h-1},1,n_1,...)$ by concatenating a finite balanced non-degenerate configuration of height $h$. 

Observe that the change of parameters in \eqref{eq:u_il_k} follows translations as well as scalar multiplication, so we need to modify the concatenation conditions as in Section 3 of ~\cite{morabito2012non}, so that the end point $p_{h,1}$ and the first point $p_{0,1}$ coincide during concatenation. Define the configuration $\boldsymbol{U}$ of type $(n_k)_{k\in \Z}$ as:
 \begin{equation}\label{eq:concatenation}
 \begin{aligned}
  \text{for, } m\in \Z, 1\leq k\leq h, \quad
l_{mh+k} :=
\begin{dcases}
    l_k \quad \text{ if }  h \text{ is odd}\,\\
    (-1)^ml_k \quad \text{if } h \text{ is even}
\end{dcases},  \\[1em]
\text{for, } m\in \Z, 1\leq k\leq h-1, 2\leq i\leq n_h,  u_{k,i} := 
\begin{dcases}
    u_{k,i} \quad \text{if h is even },\\
    (-1)^mu_{k,i} \quad \text{if } h \text{ is odd},
\end{dcases}
 \end{aligned}
\end{equation} 
which follows $l_{k+T}=l_k$ and $u_{k+T,i} = u_{k,i}$. The corresponding neck positions $(p_{k,i})_{k\in \Z,1\leq i\leq n_k }$ satisfy $p_{k+T,i}= p_{k,i} + C$ for some constant $C\in \C \setminus\{0\}$. Similarly, we can concatenate a sequence of different height finite balanced and non-degenerate configurations, yielding a periodic and quasi-periodic family of maxfaces.

\subsection{Period-2 maxfaces with swallowtail and almost conical singular points.}\label{ex:height2}
In this subsection we explain a simple height–2 example where the singularities in the two layers behave differently. For every small $t>0$ we get a maxface $M_t$ which is periodic of period $2$. In all odd layers each neck produces exactly four swallowtails, while in all even layers the unique neck has an almost conical singularity.

\medskip

Fix $n\in\mathbb{N}$. Consider the height–2 configuration of type $(1,n,1)$ given by
\[
p_{0,1}=0,\quad p_{2,1}=2\sqrt{-1},\quad p_{1,j}=\sqrt{-1} + \cot\bigl(\tfrac{j\pi}{n+1}\bigr)\quad(1\le j\le n).
\]
Since $u_{1,1}=0$, the corresponding finite configuration $\boldsymbol{U}$ is
\begin{align*}
    \boldsymbol{U} &= (l_1, u_{1,2},\dots,u_{1,n},l_2) \\
    &= \bigl(-(\sqrt{-1} + \cot(\tfrac{\pi}{n+1})),\; \cot(\tfrac{\pi}{n+1})-\cot(\tfrac{2\pi}{n+1}),\; \dots,\; \sqrt{-1} - \cot(\tfrac{\pi}{n+1})\bigr).
\end{align*}
It is proved in \cite{morabito2012non,traizet2013opening} that this configuration is balanced and non–degenerate, and moreover it is unique up to the natural permutation symmetry of the $n$ necks.

\medskip

Using the above concatenation procedure, we extend this to an infinite configuration $(p_{k,i})_{k\in \mathbb Z,\,1\le i\le n_k}$, where the layers alternate between $n$ necks and one neck; that is, $n_k=n$ for odd $k$ and $n_k=1$ for even $k$. Concretely,
\[
p_{0,1}= 0,\qquad (p_{1,i})_{1\leq i \leq n}= \sqrt{-1} + \cot\bigl(\tfrac{i\pi}{n+1}\bigr),\qquad p_{2,1}= 2\sqrt{-1},
\]
and inductively
\[
p_{k+2,i} = p_{k,i}+2\sqrt{-1} \quad \text{for all } k \in \mathbb Z.
\]
Hence, for each small $t>0$, the maxface given by Theorem~\ref{main theorem} is $2$–periodic and has infinitely many ends.

\medskip

When $n=2$, this gives the Lorentzian analogue of a genus–$1$ “Wei–type’’ surface, i.e. a $3$–ended maxface of genus $1$. For general $n$ we get higher–genus examples with the same alternating pattern of ends.

\medskip

From a direct computation one obtains, for odd $k\in \mathbb Z$, the functions $R_{k,i'}^{(1)}(\theta)$ from \eqref{eq:R_{k,i}}:
\[
R_{k,i'}^{(1)}(\theta)\; =\;\frac{6\cos2\theta}{n\bigl(1+\cot^2(\tfrac{i'\pi}{n+1})\bigr)^2}\,\Bigl(\cot\tfrac{i'\pi}{n+1} - R\Bigr) \not\equiv 0,
\]
where
\[
R= \frac{1}{n}\sum_{j=1,\neq i'}^n \frac{\sin\frac{(j-i')\pi}{n+1}}{\sin\frac{i'\pi}{n+1}\sin\frac{j\pi}{n+1}}.
\]
For each $1\le i'\le n$, the function $R_{k,i'}^{(1)}(\theta)$ has exactly four simple zeros in $\theta$, equally spaced. Each zero corresponds to a swallowtail. Hence every neck in an odd layer produces \emph{four} swallowtails.

\medskip

For the even layers (i.e. $k$ even) we have only one neck, and in this case
\begin{equation}\label{eq:R_{k,1}
}
    R_{k,1}^{(r)}(\theta) = \operatorname{Im} \Big[ e^{(r+1)\sqrt{-1}\theta} \operatorname{Res}_{p_{k,1}-p_{k-1,1}} \Big(\frac{\omega_k}{dz}\Big)^{r+2} - e^{-(r+1)\sqrt{-1}\theta} \overline{ \operatorname{Res}_0\Big(\frac{\omega_{k+1}}{dz}\Big)^{r+2}}\Big].
\end{equation}
For convenience we denote $\operatorname{Res}_0\bigl(\frac{\omega_1}{dz}\bigr)^{r+2}$ by $R$.

\begin{prop}\label{prop1}
For all $r\in \mathbb{N}$, we have $R_{0,1}^{(r)}(\theta)\equiv0$. In fact $R_{k,1}^{(r)}(\theta)\equiv0$ for all even $k\in\mathbb Z$. Moreover, $R$ is purely imaginary when $r$ is even, and $R$ is real when $r$ is odd.
\end{prop}

\begin{proof}
The proof uses the symmetry of the points $(p_{k,i})_{1\le i\le n_k}$ and the periodicity $p_{k+2,i} = p_{k,i}+2\sqrt{-1}$. Since each $p_{k,1}$ is a vertical translate of $p_{0,1}$, it is enough to work with $p_{0,1}$. From \eqref{eq:w_k}, for $k=1$,
\[
 \omega_1 \;=\;
  \frac{c_0}{z-(p_{0,1}-p_{0,1})}
\;-\;
\sum_{i=1}^{n}
  \frac{c_1}{z-(p_{1,i}-p_{0,1})},
  \qquad c_0=1,\ c_1=\frac{1}{n},\ p_{0,1}=0.
\]
Thus
\begin{align*}
R = \operatorname{Res}_0\Big(\frac{\omega_1}{dz}\Big)^{r+2}
&= \operatorname{Res}_0\Big(\frac{1}{z}-\sum_{i=1}^{n}\frac{(1/n)}{z-p_{1,i}}\Big)^{r+2} \\
&=  \operatorname{Res}_0\Big(\frac{1}{z}-\frac{1}{n}\sum_{i=1}^{n}\Big(\frac{-1}{p_{1,i}}\Big)\frac{1}{1 - \frac{z}{p_{1,i}}}\Big)^{r+2}\\
&=\operatorname{Res}_0\Big(\frac{1}{z} + \sum_{k=0}^{\infty} a_kz^k\Big)^{r+2}, \quad a_k := \frac{1}{n}\sum_{i=1}^{n}\frac{1}{(p_{1,i})^{k+1}}\\
&= \operatorname{Res}_0\Big(\frac{1}{z^{r+2}}+\binom{r+2}{1} \frac{A(z)}{z^{r+1}} +\dots+(A(z))^{r+2}\Big),
\end{align*}
where $A(z) := \sum_{k=0}^{\infty} a_kz^k$. Write $(A(z))^j = \sum_{m=0}^{\infty} c_{j,m}z^m$ with $c_{j,m}$ polynomial in the $a_k$’s.  

The $z^{-1}$–coefficient is
\begin{equation}\label{eq:R}
    R= \sum_{j=1}^{r+1} \binom{r+2}{j} c_{j,r+1-j} = \sum_{j=1}^{r+1} \binom{r+2}{j}\Bigg[\sum_{\substack{k_1,\dots,k_j\geq 0\\ k_1+\cdots+k_j=r+1-j}} a_{k_1}\cdots a_{k_j}\Bigg].
\end{equation}
A similar formula holds for $\bigl(\frac{\omega_0}{dz}\bigr)^{r+2}$, with coefficients $b_k= \frac{1}{n}\sum_{i=1}^{n}\frac{1}{(p_{-1,i})^{k+1}}$. By symmetry $b_k= \overline{a_k}$, so $\operatorname{Res}_{-p_{-1,i}}\bigl(\frac{\omega_k}{dz}\bigr)^{r+2} = (-1)^{r+2}\overline{R}$. Plugging this into \eqref{eq:R_{k,1}} for $k=0$ gives
\[
R_{0,1}^{(r)}(\theta) = \operatorname{Im} \Big[ e^{\sqrt{-1}(r+1)\theta} (-1)^{r+2}\bar{R}  - e^{-\sqrt{-1}(r+1)\theta}\bar{R} \Big].
\]
To finish, we show $R$ is purely imaginary for $r$ even and real for $r$ odd, exactly as in the original argument, using the pairing
\[
(p_{1,n-j})_{0\le j\le \frac{n}{2}}= -\overline{(p_{1,i})_{1\le i\le \frac{n}{2}}}
\]
when $n$ is even (and the similar statement when $n$ is odd). This gives the parity properties of $a_k$ and hence of $R$; substituting back shows $R_{0,1}^{(r)}(\theta)\equiv 0$ for all $r$.
\end{proof}

By Definition~\ref{def:conical}, this shows that all even layers have almost conical singularities, while all odd layers have four swallowtails on each neck. Thus this provides an example with infinitely many swallowtails and infinitely many almost conical singularities.

\subsection{Period-3 maxfaces with swallowtails in every neck.}\label{ex:height3}
We now describe a height–3 example in which \emph{every} neck produces swallowtail singularities.

Consider the height–3 configuration of type \((1,2,2,1)\) given by
\[
\begin{aligned}
p_{0,1}&=0,\qquad  p_{3,1}=3\sqrt{-1},\\
p_{1,1}&=-\tfrac{\sqrt2}{2}+ \sqrt{-1},\qquad p_{1,2}=\tfrac{\sqrt2}{2}+ \sqrt{-1},\\
p_{2,1}&=-\tfrac{\sqrt2}{2}+ 2\sqrt{-1},\quad\; p_{2,2}=\tfrac{\sqrt2}{2}+ 2\sqrt{-1}.
\end{aligned}
\]
With these neck positions the finite configuration is balanced and non–degenerate (as in the previous height–2 case).

\medskip

By applying the same concatenation procedure as before, we obtain an infinite configuration
\[
\boldsymbol{U} = (\dots,2,1,2,2,1,2,\dots),
\]
i.e. the sequence of numbers of necks is periodic of period \(3\): one layer with \(2\) necks, the next with \(1\) neck, then again two layers with \(2\) necks, and so on. More precisely, the parameters satisfy
\[
l_{k+3} = l_k \quad\text{and}\quad u_{k+3,i} = u_{k,i} \qquad (k\in \mathbb Z,\; 2\le i \le n_k),
\]
so we get a $3$–periodic balanced and non–degenerate configuration. In particular,
\[
\boldsymbol{U} = (\dots,\, l_1=\tfrac{1}{\sqrt{2}}-\sqrt{-1},\; u_{1,2} =-\sqrt{2},\; l_2= \sqrt{-1},\; u_{2,2}=\sqrt{2},\; l_3= -(\tfrac{1}{\sqrt{2}}+\sqrt{-1}),\; l_4,\dots),
\]
and hence, by Theorem~\ref{main theorem}, this produces a complete, embedded (in the wider sense) maxface of even genus, periodic of period \(3\). The associated neck positions \((p_{k,i})_{k\in\mathbb Z,\,1\le i\le n_k}\) are exactly the above ones for \(k=0,1,2,3\), and for all other layers are determined by the vertical periodicity
\[
p_{k+3,i}= p_{k,i} +3\sqrt{-1} \qquad (k\in \mathbb Z).
\]

\medskip
\noindent\textbf{Singularities.}
This configuration has a vertical reflection symmetry (the $yz$–plane) which passes through all the necks \((p_{3k,1})_{k\in \mathbb Z}\). By Proposition~5.4 of \cite{chennode}, at the fixed points of this vertical reflection the singularities are automatically non–cuspidal. A direct (but slightly long) computation of the functions \(R_{k,i}^{(1)}(\theta)\) in \eqref{eq:R_{k,i}} at each neck gives:
\[
\begin{aligned}
    R_{0,1}^{(1)}(\theta) &= \tfrac{1}{6}(1.4\sin2\theta +2.4\sqrt{2}\cos2\theta), & R_{3,1}^{(1)}(\theta)&=\tfrac{1}{6}(1.4\sin2\theta - 2.4\sqrt{2}\cos2\theta),\\
    R_{1,1}^{(1)}(\theta) &= \frac{3}{8}(0.6\sqrt2\cos2\theta+ 0.16\sin2\theta), & R_{1,2}^{(1)}(\theta)&=\tfrac{3}{8}(0.05\sqrt{2}\cos2\theta- 1.2\sin2\theta),\\
    R_{2,1}^{(1)}(\theta) &= \tfrac{1}{384}(108.2\sin2\theta + 85.3\sqrt{2}\cos2\theta), & 
    R_{2,2}^{(1)}(\theta)&=\tfrac{1}{384}(142.3\sin2\theta - 10.8\sqrt{2}\cos2\theta).
\end{aligned}
\]
(Here we are simply writing the coefficients exactly as they come out of the computation.)

Notice that \emph{none} of the above six functions is identically zero. Hence, by Theorem~\ref{thm:singular-structure}, each of the six necks appearing in the first three layers (from $k=0$ to $k=3$) has exactly four swallowtail singularities for sufficiently small $t>0$. Thus in one period we see $
6 \times 4 \;=\; 24$
swallowtail singularities, and all the other points on the singular curves are cuspidal edges. Because the configuration is $3$–periodic, this pattern repeats for all $k\in\mathbb Z$, so the surface has infinitely many swallowtails.

This height–3 construction is useful because it gives a very symmetric maxface of even genus, with infinitely many planar ends, and with the singularities completely under control: every neck produces the “four swallowtails’’ pattern.

\medskip

\subsection{Maxfaces with an almost conical singularity in every neck}
We also record the opposite extreme. Take any height \(h\) and consider the configuration
\[
p_{k,1} = ka \qquad (0 \le k \le h),
\]
with a fixed $a\in \mathbb C\setminus\{0\}$, and of type
\[
(n_0=1,\,1,\,\dots,\,1,\,n_h=1),
\]
so that each layer has exactly one neck. By \eqref{eq:concatenation}, the corresponding sequence \(\boldsymbol{U}\) satisfies
\[
l_{k+2} = l_k,
\]
hence it is balanced and non–degenerate and gives a $2$–periodic maxface. For this family we have, for \emph{every} neck and for \emph{every} order $r\in\mathbb N$,
\[
R^{(r)}_{k,1}(\theta) \equiv 0 \qquad (k\in \mathbb Z),
\]
so each neck carries an almost conical singularity (in the sense of Definition~\ref{def:conical}). This family of examples is very close, in its appearance, to the classical Riemann–type maximal surfaces constructed in \cite{lopez2000maximal}.

\section*{Statements and Declarations}

\subsection*{Funding}
This work was supported by a Postdoctoral Fellowship from the International Centre for Theoretical Sciences, Bengaluru, India.


\subsection*{Conflict of Interest}
The author declares that there are no conflicts of interest regarding the publication of this paper.

\subsection*{Data Availability}
No datasets were generated or analyzed during the current study.

\bibliography{ref.bib}

\end{document}